\documentclass{amsart}
\usepackage[utf8]{inputenc}

\usepackage{amsmath}
\usepackage{amsfonts}
\usepackage{amsthm}
\usepackage{cancel}
\usepackage{enumitem}
\usepackage{hyperref}
\usepackage{float}
\usepackage{tikz}
\usepackage{comment}
\usetikzlibrary{trees}
\usepackage{thmtools}
\usepackage{thm-restate}
\usepackage{hyperref}
\usepackage{cleveref}
\usetikzlibrary{shapes.geometric}

\newtheorem{theorem}{Theorem}[section]
\newtheorem{definition}[theorem]{Definition}
\newtheorem{lemma}[theorem]{Lemma}

\newtheorem{conjecture}[theorem]{Conjecture}

\newtheorem{notation}[theorem]{Notation}
\newtheorem{example}[theorem]{Example}

\usetikzlibrary{positioning, quotes}
\newdimen\nodeDist
\tikzset{
    position/.style args={#1:#2 from #3}{
        at=(#3.#1), anchor=#1+180, shift=(#1:#2)
    }
}

\tikzstyle{every node} = [circle, fill=gray!30,inner sep=0pt, minimum size = .5cm]
\tikzstyle{label} = [fill=white!0]
\tikzstyle{edge} = [draw=black]
\tikzstyle{edge_recent} = [draw=black, line width=2, very thick]
\tikzstyle{new} = [fill=red!30]
\tikzstyle{green} = [fill=green!30]
\tikzstyle{purple} = [fill=purple!30]
\tikzstyle{cyan} = [fill=cyan!30]
\tikzstyle{orange} = [fill=orange!30]
\tikzstyle{blue} = [fill=blue!30]

\title{Conditions for Building Generalized Action Graphs from Sequences}
\author[S. Klanderman]{Sarah Klanderman}
\address{Department of Mathematical and Computational Sciences, Marian University, 3200 Cold Spring Road, Indianapolis, IN, 46222, USA}
\email{sklanderman@marian.edu}

\author[K. McDicken]{Katy McDicken}
\address{Department of Mathematical and Computational Sciences, Marian University, 3200 Cold Spring Road, Indianapolis, IN, 46222, USA}
\email{kmcdicken936@marian.edu}

\author[A. Tebbe]{Amelia Tebbe}
\address{Department of Mathematics, Indiana University Kokomo, 2300 S. Washington St., Kokomo, IN, 46902, USA}
\email{antebbe@iu.edu}

\date{}



\begin{document}
\maketitle

\section{Abstract}
This paper explores the properties of directed graphs, termed \emph{generalized action graphs}, which exhibit a strong connection to certain number sequences. Focusing on the structural and combinatorial aspects, we investigate the conditions under which specific sequences can generate generalized action graphs. Building upon prior research in this field, we analyze specific features of these graphs and how they correspond to patterns and properties in their sequences. These findings support a broader conclusion that establishes framework for identifying which sequences can produce generalized action graphs. 
\section{Introduction} 
Action graphs were first constructed in work studying rooted category actions by Bergner-Hackney. Alvarez et al. later showed that action graphs can be defined inductively and are related to the Catalan numbers, a well known sequence of natural numbers that have many ties to combinatorics \cite{ActionGraphs}. Caldwell et al. also observed that each \emph{action graph} has subgraphs that are isomorphic to previous graphs. These graphs were later generalized  by Cressman-Lin-Nguyen-Wiljanen, and Caldwell et al., showing a different set of directed graphs are similarly related to the Fuss-Catalan numbers and the super Catalan numbers, respectively \cite{CURMGroup}. While previous research has established properties of action graphs, a generalized framework for determining which sequences generate such graphs remains elusive. In this paper, we establish criteria for identifying sequences that yield action graphs. This result unifies and extends previous findings in this area. Formally, we will prove the following theorem: 
\begin{restatable}{thm}{Mtheorem}\label{Mtheorem}
    Consider a positive sequence $s_n$ with $n\geq 0$ and $s_0=1$. If there exist positive integers $z_n$ such that $z_1=s_1$, and for all $n \geq 2$,  $$s_n = z_n + \sum_{i=1}^{n-1} z_i \cdot s_{n-i},$$  then there exist generalized action graphs $G_n$ and, for $i\leq n$, $z_i$ is the number of vertices in $G_n$ labeled $i$ adjacent to the root.  
\end{restatable}

This theorem provides a method for determining which sequences will yield generalized action graphs. Furthermore, the proof of the theorem describes a method for constructing the graphs.

In this paper, we summarize previous work constructing graphs relating to the Catalan, Fuss-Catalan, and super Catalan number sequences and compare those sequences with our theorem. Then, in Section \ref{Sec: AGP}, we consider the properties that all action graphs satisfy. Section \ref{Sec: BAG} provides the proof of Theorem \ref{Mtheorem}.

\section{Generalized Action Graphs}
Generalized action graphs have been constructed to model the Catalan numbers, the Fuss-Catalan numbers, and the super Catalan numbers. 
The process of generating such directed graphs in each of these settings is unique to the sequence that it represents. Alvarez, Bergner, and Lopez were the first to relate these action graphs to the Catalan numbers \cite{ActionGraphs}. 
\subsection{Catalan Numbers}
\begin{definition}[\cite{catalantext}] \label{Catdef}
    The \emph{Catalan numbers}, denoted $C_n$, are a sequence of natural numbers given by 
$$ C_0 = 1, C_{n} = \displaystyle\sum_{i=0}^{n-1} C_i C_{n-1-i} = \binom{2n}{n} \frac{1}{n+1}.$$
\end{definition}
Observe that the first few numbers of this sequence are: $$[1, 1, 2, 5, 14 \ldots].$$ In \cite{ActionGraphs}, the authors showed that the Catalan numbers count the number of leaves added to each action graph. 
\begin{definition}[\cite{ActionGraphs}]
    Define the action graph $A_{n+1}$ inductively, starting with one vertex labeled zero for $A_0$. Construct an edge to a new vertex labeled $n+1$ from each leaf in $A_n$ and from the source of all non-trivial paths to vertices labeled $n$ in $A_n$.
\end{definition}
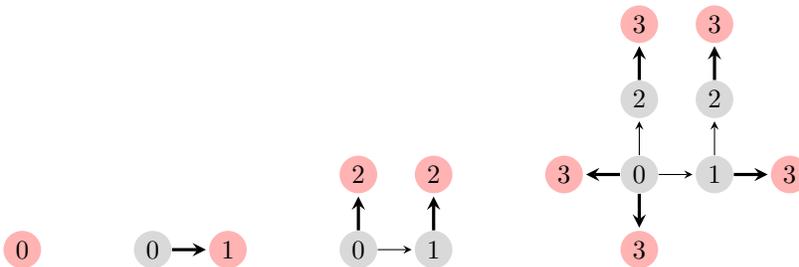
\begin{figure} [h]
\begin{center}
\begin{tikzpicture}[> = stealth,shorten > = 1pt,scale=.5]
    \node[new] (root) at (0,0) {0};
\end{tikzpicture}
\hspace{1cm}
\begin{tikzpicture}[> = stealth,shorten > = 1pt,scale=.5]
    \node (root) at (0,0) {0};
    \node[new] (1) at (2,0) {1};
    \path[->][edge_recent]  (root) to (1);
\end{tikzpicture}
\hspace{1cm}
\begin{tikzpicture}[> = stealth,shorten > = 1pt,scale=.5]
    \node (root) at (0,0) {0};
    \node (1) at (2,0) {1};
    \node[new] (2) at (2,2) {2};
    \node[new] (3) at (0,2) {2};
    \path[->][edge]  (root) to (1);
    \path[->][edge_recent]  (root) to (3);
    \path[->][edge_recent]  (1) to (2);
\end{tikzpicture}
\hspace{1cm}
\begin{tikzpicture}[> = stealth,shorten > = 1pt,scale=.5]
    \node (root) at (0,0) {0};
    \node (1) at (2,0) {1};
    \node (2) at (2,2) {2};
    \node (3) at (0,2) {2};
    
    \node[new] (4) at (0,4) {3};
    \node[new] (5) at (2,4) {3};
    \node[new] (6) at (4,0) {3};
    \node[new] (7) at (-2,0) {3};
    \node[new] (8) at (0,-2) {3};
    
    \path[->][edge]  (root) to (1);
    \path[->][edge]  (root) to (3);
    \path[->][edge]  (1) to (2);
    \path[->][edge_recent]  (3) to (4);
    \path[->][edge_recent]  (2) to (5);
    \path[->][edge_recent]  (1) to (6);
    \path[->][edge_recent]  (root) to (7);
    \path[->][edge_recent]  (root) to (8);
\end{tikzpicture}
\end{center}
\caption{Action Graphs $A_0$ through $A_3$}
\label{actiongraphfigure}
\end{figure}

The vertices are labeled starting at $0$, and every vertex added to each graph will be labeled with the same number, corresponding to the step in the inductive definition. For example, new vertices in action graph $A_2$ are labeled $2$, and there will be a total of $C_2$ vertices added to construct that graph.

Since the Catalan numbers have been shown to yield action graphs, and in fact are encoded in the original example of action graphs, we will consider how this result compares to Theorem \ref{Mtheorem} in the following example. More specifically, we examine $A_3$, given in the figure below.
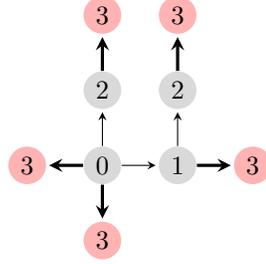
\begin{figure} [H]
\begin{center}
\hspace{1cm}
\begin{tikzpicture}[> = stealth,shorten > = 1pt,scale=.5]
    \node (root) at (0,0) {0};
    \node (1) at (2,0) {1};
    \node (2) at (2,2) {2};
    \node (3) at (0,2) {2};
    
    \node[new] (4) at (0,4) {3};
    \node[new] (5) at (2,4) {3};
    \node[new] (6) at (4,0) {3};
    \node[new] (7) at (-2,0) {3};
    \node[new] (8) at (0,-2) {3};
    
    \path[->][edge]  (root) to (1);
    \path[->][edge]  (root) to (3);
    \path[->][edge]  (1) to (2);
    \path[->][edge_recent]  (3) to (4);
    \path[->][edge_recent]  (2) to (5);
    \path[->][edge_recent]  (1) to (6);
    \path[->][edge_recent]  (root) to (7);
    \path[->][edge_recent]  (root) to (8);
\end{tikzpicture}
\end{center}
\caption{Action Graph $A_3$}
\label{actiongraphA3}
\end{figure}
\begin{example} \label{ex3.3}
    From Definition \ref{Catdef}, $C_n=[1, 1, 2, 5, 14 \ldots]$. Taking a look at Figure \ref{actiongraphA3}, we can determine the $z_n$ values of Theorem \ref{Mtheorem} by counting the number of vertices of each label coming off the root vertex. For action graph $A_3$, $z_n = [1,1,2\dots]$. We see that these values satisfy the formula of Theorem \ref{Mtheorem}: 
    \begin{align}
        \notag
        C_3 &= z_3 + z_1 \cdot C_2 + z_2 \cdot C_1 \\ \notag
        &= 2 + 1 \cdot 2 + 1 \cdot 1 \\ \notag
        &= 5 
    \end{align}
    Since $C_3 = 5$, Theorem \ref{Mtheorem} asserts that Catalan numbers can be modeled using action graphs, at least up to $n=3$, which is consistent with the work of \cite{ActionGraphs}. \\
    We can see this through another example using action graph $A_4$. Through construction, we can determine the $z_n$ values. For action graph $A_4$, $z_n = [1,1,2,5 \ldots]$. We see that these values satisfy the formula of Theorem \ref{Mtheorem}:
    \begin{align}
        C_4 &= z_4 + z_1 \cdot C_3 + z_2 \cdot C_2 + z_3 \cdot C_1 \\ \notag
        &= 5 + 1 \cdot 5 + 1 \cdot 2 + 2 \cdot 1 \\ \notag 
        &= 14
    \end{align}
    This pattern continues for all $n$.
\end{example}

There is a further connection between the values $z_j$ of Theorem \ref{Mtheorem} and the Catalan number sequence values $C_j$.

\begin{theorem}
    For the Catalan number sequence, $C_j$, and the corresponding sequence $z_j$ of Theorem \ref{Mtheorem},  $$z_j = C_j - \sum_{i=1}^{j-1}{z_i \cdot C_{j-i}}= C_{j-1}.$$
\end{theorem}
\begin{proof}
    We will use the principle of mathematical induction. Let $C_j$ be the Catalan numbers and $z_j$ be the number of vertices labeled $j$ connected to the vertex labeled $0$. Observe by rearranging the formula from Theorem \ref{Mtheorem}:
    $$z_j = C_j - \sum_{j-1}^{i=1}{z_j \cdot C_{j-i}}$$
    It remains to show that $z_j=C_{j-1}.$
    
\underline{Base Case:} \\
     From Example \ref{ex3.3}, we know $z_1 = 1$, which is also equal to $C_0$.  
    
\underline{Inductive Hypothesis:}
    Assume for all $1\leq j<k$\\
    \begin{align*}
        z_j &= C_{j-1}.
    \end{align*}
    \underline{Inductive Step:} We will prove that $z_k = C_{k-1}$. As previously noted, $$z_k = C_k - \sum_{i=1}^{k-1}{z_i \cdot C_{k-i}}.$$
    Below, we expand this summation, use the summation definition of Catalan numbers to replace $C_k$ with a sum, and then cancel corresponding terms. 
    \begin{align*}
        z_k &= C_k - z_1 C_{k-1} - z_2  C_{k-2} - z_3  C_{k-3} -\cdots -z_{k-2}  C_2 - z_{k-1}  C_1 \\
        &= C_k - C_0  C_{k-1} - C_1  C_{k-2} - C_2  C_{k-3} - \cdots - C_{k-3}  C_2 - C_{k-2}  C_1  \\
        &= (C_0  C_{k-1} + C_1  C_{k-2} + C_2  C_{k-3} + \cdots + C_{k-3}  C_2 + C_{k-2} C_1 + C_{k-1}  C_0)  \\
        &\hspace{2\parindent} - C_0  C_{k-1} - C_1  C_{k-2} - C_2  C_{k-3} - \cdots - C_{k-3} C_2 - C_{k-2}  C_1 \\
        &= (C_0  C_{k-1} + \cancel{C_1  C_{k-2}} + \cancel{C_2  C_{k-3}} + \cdots + \cancel{C_{k-3}  C_2} + \cancel{C_{k-2} C_1} + \cancel{C_{k-1}  C_0})  \\
        &\hspace{2\parindent} - \cancel{C_0  C_{k-1}} - \cancel{C_1  C_{k-2}} - \cancel{C_2  C_{k-3}} - \cdots - \cancel{C_{k-3} C_2} - \cancel{C_{k-2}  C_1} \\
        &= C_{k-1}C_0 
    \end{align*}
    Since $C_0 = 1$, the result is
    \begin{align*}
        z_k &= C_{k-1}.
    \end{align*}
    Therefore $z_k = C_{k-1}$ for all $k>0$.
\end{proof}

\subsection{Path Length} 
Two important aspects of constructing generalized action graphs are \emph{\textit{paths}} and \emph{\textit{path length}}. For our purposes, throughout this paper \textit{paths} will refer to directed paths that start at any vertex and end at a leaf. Note that the longest paths in the graph $A_{n+1}$ will start at the vertex labeled $0$ and end at a vertex labeled $n+1$. We will denote path length by $\ell$, defined as the number of edges traveled from the initial vertex to the final vertex. 
\begin{figure} [H]
\begin{center}
\hspace{1cm}
\begin{tikzpicture}[> = stealth,shorten > = 1pt,scale=.5]
    \node[cyan, diamond] (root) at (0,0) {0};
    \node[rectangle, fill=none] at (2, -1) {\tiny{$\ell = 3$}};
    \node[cyan, diamond] (1) at (2,0) {1};
    \node[cyan, diamond] (2) at (2,2) {2};
    \node (3) at (0,2) {2};
    
    \node[orange, semicircle] (4) at (0,4) {3};
    \node[rectangle, fill=none] at (-1.65, 4) {\tiny{$\ell=0$}};
    \node[cyan, diamond] (5) at (2,4) {3};
    \node (6) at (4,0) {3};
    \node (7) at (-2,0) {3};
    \node (8) at (0,-2) {3};
    
    \path[->][edge]  (root) to (1);
    \path[->][edge]  (root) to (3);
    \path[->][edge]  (1) to (2);
    \path[->][edge]  (3) to (4);
    \path[->][edge]  (2) to (5);
    \path[->][edge]  (1) to (6);
    \path[->][edge]  (root) to (7);
    \path[->][edge]  (root) to (8);
\end{tikzpicture}
\hspace{1cm}
\begin{tikzpicture}[> = stealth,shorten > = 1pt,scale=.5]
    \node[new, isosceles triangle] (root) at (0,0) {0};
    \node[rectangle, fill=none] at (2, -1) {\tiny{$\ell=2$}};
    \node[green, rectangle] (1) at (2,0) {1};
    \node[green, rectangle] (2) at (2,2) {2};
    \node (3) at (0,2) {2};
    
    \node (4) at (0,4) {3};
    \node[green, rectangle] (5) at (2,4) {3};
    \node[rectangle, fill=none] at (-1, 1) {\tiny{$\ell=1$}};
    \node (6) at (4,0) {3};
    \node[new, isosceles triangle] (7) at (-2,0) {3};
    \node (8) at (0,-2) {3};
    
    \path[->][edge]  (root) to (1);
    \path[->][edge]  (root) to (3);
    \path[->][edge]  (1) to (2);
    \path[->][edge]  (3) to (4);
    \path[->][edge]  (2) to (5);
    \path[->][edge]  (1) to (6);
    \path[->][edge]  (root) to (7);
    \path[->][edge]  (root) to (8);
\end{tikzpicture}
\end{center}
\caption{Different Paths of Length $\ell$ on Graph $A_3$}
\label{actiongraphfigurepathlength}
\end{figure}
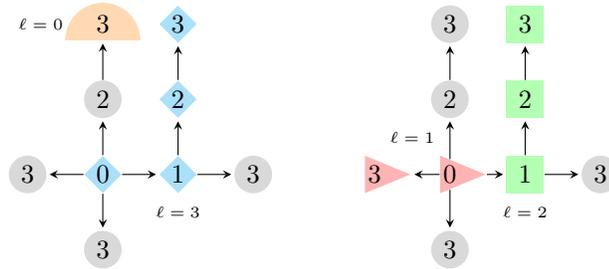
Figure \ref{actiongraphfigurepathlength} shows two versions of action graph $A_3$, colored in different ways to highlight different paths. 
As seen above, there are four colored paths on graph $A_3$, of varying path length, $\ell$. The graph on the left shows a path of length 3, starting from the 0 vertex, highlighted by \textcolor{cyan}{blue diamonds}. A trivial path length of 0, starting and ending at a vertex labeled 3 is highlighted by an \textcolor{orange}{orange semicircle}. The graph on the right shows a path length of 1 in \textcolor{red}{red triangles}, starting at the vertex labeled 0. There is also a path length of 2, of \textcolor{green}{green rectangles}, starting from the 1 vertex. Path length does not impact the construction of the action graphs related to the Catalan numbers; however, it plays a large role in constructing the generalized action graphs related to other number sequences. \\

\subsection{Fuss-Catalan Numbers}
The \emph{Fuss-Catalan} numbers are related to the Catalan numbers, and they can also be modeled using a generalization of the action graph construction. 
\begin{definition}\label{FussCatDef}(\cite{catalantext})
    The \emph{Fuss-Catalan} numbers are a generalization of the Catalan numbers with two arguments, $n$ and $k$. There are both recursive and explicit formulas for the Fuss-Catalan numbers:
    \[  C_{n,k} = \sum_{n_1+n_2+\ldots+n_{k+1}=n-1} \prod_{i=1}^{k+1} C_{n_i,k} = \frac{\binom{n(k+1)}{n}}{kn+1}.\] 
\end{definition}
Unlike the Catalan numbers, the Fuss-Catalan numbers are dependent upon 2 variables. When $k=2$, the valuse of the Fuss-Catalan numbers are:  $$C_{n,2} = [1, 1, 3, 12, 55 \ldots].$$ And when $k=1$: $$C_{n,1} = C_n.$$\\
The authors in \cite{GenActionGraphs} expanded on the work done previously by Alvarez, Bergner, and Lopez by developing new action graphs for the Fuss-Catalan numbers, which they called \textit{generalized action graphs}. \\
\begin{definition}[\cite{GenActionGraphs}]
The \emph{generalized action graph} $T_{n,k}$ is generated by adding $\binom{\ell + k - 1}{\ell}$ vertices to each vertex of the previous graph, where $\ell$ is the path length. The first graph, $A_0$ consists of one vertex labeled 0. The total number of new vertices is equal to the Fuss-Catalan number $C_{n,k}$.
\end{definition}

\begin{figure}[H]
\begin{center}

\begin{tikzpicture}[> = stealth,shorten > = 1pt,scale=.5]
    \node[new] (root) at (0,0) {0};
    \node[rectangle, fill=none] at (0, -1) {\tiny{$C_{0,2}=1$}};
\end{tikzpicture}
\hspace{1cm}
\begin{tikzpicture}[> = stealth,shorten > = 1pt,scale=.5]
    \node (root) at (0,0) {0};
    \node[rectangle, fill=none] at (0.5, -1) {\tiny{$C_{1,2}=1$}};
    \node[new] (1) at (2,0) {1};
    \path[->][edge_recent]  (root) to (1);
\end{tikzpicture}
\hspace{1cm}
\begin{tikzpicture}[> = stealth,shorten > = 1pt,scale=.5]
    \node (root) at (0,0) {0};
    \node[rectangle, fill=none] at (0, -1) {\tiny{$C_{2,2}=3$}};
    \node (1) at (2,0) {1};
    \node[new] (2) [position=80:{\nodeDist} from root] {2};
    \node[new] (3) [position=120:{\nodeDist} from root] {2};
    \node[new] (4) [above of=1] {2};
    \path[->][edge]  (root) to (1);
    \path[->][edge_recent]  (root) to (3);
    \path[->][edge_recent]  (root) to (2);
    \path[->][edge_recent]  (1) to (4);
\end{tikzpicture}
\hspace{1cm}
\begin{tikzpicture}[> = stealth,shorten > = 1pt,scale=.25]
    \node (root) at (-0.5,0) {0};
    \node[rectangle, fill=none] at (0, -8) {\tiny{$C_{3,2}=12$}};
    \node (1) at (4,0) {1};
    
    \node (2) [position=80:{\nodeDist} from root] {2};
    \node (3) [position=120:{\nodeDist} from root] {2};
    \node (4) [above of=1] {2};
    
    \node[new] (5) [above of=2] {3};
    \node[new] (6) [above of=3] {3};
    \node[new] (7) [above of=4] {3};
    \node[new] (8) [above right of=1] {3};
    \node[new] (9) [right of=1] {3};
    
    \node[new] (16) [position=150:{\nodeDist} from root] {3};
    \node[new] (10) [position=180:{\nodeDist} from root] {3};
    \node[new] (11) [position=210:{\nodeDist} from root] {3};
    \node[new] (12) [position=240:{\nodeDist} from root] {3};
    \node[new] (13) [position=270:{\nodeDist} from root] {3};
    \node[new] (14) [position=300:{\nodeDist} from root]{3};
    \node[new] (15) [position=330:{\nodeDist} from root] {3};

    \path[->][edge]  (root) to (1);
    
    \path[->][edge]  (root) to (3);
    \path[->][edge]  (root) to (2);
    \path[->][edge]  (1) to (4);
    
    \path[->][edge_recent]  (2) to (5);
    \path[->][edge_recent]  (3) to (6);
    \path[->][edge_recent]  (1) to (8);
    \path[->][edge_recent]  (1) to (9);
    \path[->][edge_recent]  (4) to (7);
    
    \path[->][edge_recent]  (root) to (10);
    \path[->][edge_recent]  (root) to (11);
    \path[->][edge_recent]  (root) to (12);
    \path[->][edge_recent]  (root) to (13);
    \path[->][edge_recent]  (root) to (14);
    \path[->][edge_recent]  (root) to (15);
    \path[->][edge_recent]  (root) to (16);

\end{tikzpicture}
\end{center}
\caption{Generalized Action Graphs for the Fuss Catalan numbers $T_{0,2}$ through $T_{3,2}$}
\end{figure}
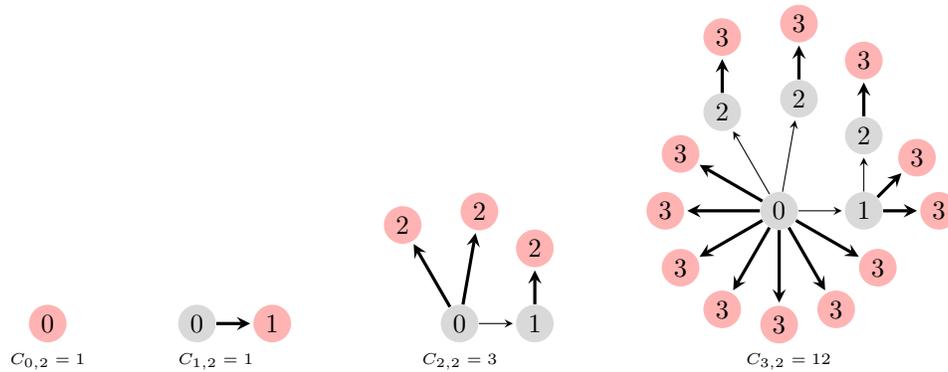
Unlike the action graphs that are related to the Catalan numbers, the number of new vertices added to each vertex for the \textit{generalized action graphs} is $\binom{\ell+k-1}{\ell}$, which is dependent on path length.
Note that by definition $$\binom{\ell+k-1}{\ell} = \frac{(\ell+k-1)!}{\ell!(k-1)!}.$$
When $k=2$, this expression simplifies to
\begin{align}
    \notag \frac{(\ell+k-1)!}{\ell!(k-1)!} = \frac{(1+2-1)!}{1!(2-1)!} &= \frac{2!}{1!(1!)} =2,
\end{align}
in terms of $\ell$. \\
So, when $k=2$, the number of vertices added to the start of each path when $\ell=0,1,2,3\ldots$ is $[1, 2, 3, 4\ldots]$.
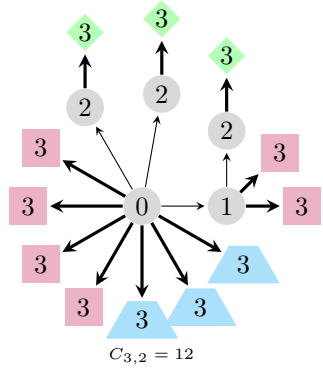
\begin{figure}[H]
\begin{center}
\begin{tikzpicture}[> = stealth,shorten > = 1pt,scale=.25]
    \node (root) at (-0.5,0) {0};
    \node[rectangle, fill=none] at (0, -8) {\tiny{$C_{3,2}=12$}};
    \node (1) at (4,0) {1};
    
    \node (2) [position=80:{\nodeDist} from root] {2};
    \node (3) [position=120:{\nodeDist} from root] {2};
    \node (4) [above of=1] {2};
    
    \node[green, diamond] (5) [above of=2] {3};
    \node[green, diamond] (6) [above of=3] {3};
    \node[green, diamond] (7) [above of=4] {3};
    \node[purple, rectangle] (8) [above right of=1] {3};
    \node[purple, rectangle] (9) [right of=1] {3};
    
    \node[purple, rectangle] (16) [position=150:{\nodeDist} from root] {3};
    \node[purple, rectangle] (10) [position=180:{\nodeDist} from root] {3};
    \node[purple, rectangle] (11) [position=210:{\nodeDist} from root] {3};
    \node[purple, rectangle] (12) [position=240:{\nodeDist} from root] {3};
    \node[cyan, trapezium] (13) [position=270:{\nodeDist} from root] {3};
    \node[cyan, trapezium] (14) [position=300:{\nodeDist} from root]{3};
    \node[cyan, trapezium] (15) [position=330:{\nodeDist} from root] {3};

    \path[->][edge]  (root) to (1);
    
    \path[->][edge]  (root) to (3);
    \path[->][edge]  (root) to (2);
    \path[->][edge]  (1) to (4);
    
    \path[->][edge_recent]  (2) to (5);
    \path[->][edge_recent]  (3) to (6);
    \path[->][edge_recent]  (1) to (8);
    \path[->][edge_recent]  (1) to (9);
    \path[->][edge_recent]  (4) to (7);
    
    \path[->][edge_recent]  (root) to (10);
    \path[->][edge_recent]  (root) to (11);
    \path[->][edge_recent]  (root) to (12);
    \path[->][edge_recent]  (root) to (13);
    \path[->][edge_recent]  (root) to (14);
    \path[->][edge_recent]  (root) to (15);
    \path[->][edge_recent]  (root) to (16);

\end{tikzpicture}
\end{center}
\caption{Generalized Action Graph $T_{3,2}$ Colored Based on Path Lengths}
\label{GenActionGraphPathLength}
\end{figure}
Figure \ref{GenActionGraphPathLength} illustrates how each vertex is added to the start of each path, based on path length. The \textcolor{green}{green diamond} vertices are added based on path lengths of 0, where one vertex is added per path. The \textcolor{purple}{purple rectangle} vertices are added for path lengths of $\ell=1$. These vertices are added to the initial vertex of the path, with two vertices being added to the vertex labeled 1, and four vertices added to the 0 vertex, for the two paths of $\ell=1$ ending at the vertices labeled 2. The \textcolor{cyan}{blue trapezoid} vertices represent the vertices added when $\ell=3$. Since there is only one path in this graph of length two, three vertices are added to the 0 vertex. 

The work of Cressman et al. constructed generalized action graphs for the Fuss-Catalan numbers. The following example shows that the Fuss-Catalan numbers also satisfy the hypotheses of Theorem \ref{Mtheorem}, demonstrated using the generalized action graph for $C_{3,2}$.
\begin{example}
    Consider the case where $k=2$. Observe from Figure \ref{GenActionGraphPathLength}, $z_n = [{1, 2, 7\ldots}]$, and from Definition \ref{FussCatDef}, our sequence values are $s_n=C_{n,2}=[1, 1, 3, 12\ldots]$. Then to satisfy the hypothesis of Theorem \ref{Mtheorem}, we must check that $$C_{n, 2} = z_n + \sum_{i=1}^{n-1} z_i \cdot C_{(n-i),2}.$$ Consider $n = 3:$
   \begin{align}
       \notag C_{3, 2} &=z_3 + z_1 \cdot C_{2,2} + z_2 \cdot C_{1,2} \\ \notag
       &= 7 + 1 \cdot 3 + 2 \cdot 1 \\ \notag
       &= 12.
   \end{align}
   Since $C_{3,2} = 12$, Theorem \ref{Mtheorem} asserts that the Fuss-Catalan numbers yield generalized action graphs for $k=2$ and $n\leq 3$. 
\end{example}

\subsection{Super Catalan Numbers}
The \emph{super Catalan} numbers are another generalization of the sequence of Catalan numbers, with two arguments, $m$ and $n$. Caldwell et al. conjectured corresponding generalized action graphs could be constructed for some cases of the super Catalan numbers \cite{CURMGroup}.\\
\\
\begin{definition}[\cite{catalantext}]\label{SuperCatalnDef}
    The super Catalan numbers are defined by $$S(m,n) = \frac{(2m)!(2n)!}{m!n!(m+n)!}.$$
\end{definition}
When $m = 0$, the first few numbers in the sequence are $$[1, 2, 6, 20, 70 \ldots].$$
Similar to the generalized action graphs for the Fuss-Catalan numbers, the authors in \cite{CURMGroup} conjectured a construction of generalized action graphs for the super Catalan numbers. The construction was based on path length as well as the number of paths. 
\begin{definition}[\cite{CURMGroup}]\label{SuperCatAGConstruction}
    Construct the sequence of directed graphs $\{G_n\}$ inductively. The graph $G_0$ is a single vertex labeled 0. To construct $G_{n+1}$ from $G_n$, consider each vertex $v$ in $G_n$. For each $0\leq\ell\leq n$, add $\displaystyle p(v,\ell)\frac{2}{2^\ell}$ new vertices labeled $n$ with edges from $v$, where $p(v,\ell)$ is the number of paths of length $\ell$ in graph $G_n$ from $v$ to a vertex labeled $n$. 
\end{definition}

\begin{figure}[H]
\begin{center}

\begin{tikzpicture}[> = stealth,shorten > = 1pt,scale=.5]
    \node[new] (root) at (0,0) {0};
\end{tikzpicture}
\hspace{1cm}
\begin{tikzpicture}[> = stealth,shorten > = 1pt,scale=.5]
    \node (root) at (0,0) {0};
    \node[new] (1) at (2,0) {1};
    \node[new] (2) at (0,2) {1};
    \path[->][edge_recent]  (root) to (1);
    \path[->][edge_recent]  (root) to (2);
\end{tikzpicture}
\hspace{1cm}
\begin{tikzpicture}[> = stealth,shorten > = 1pt,scale=.5]
    \node (root) at (0,0) {0};
    \node (1) at (2,0) {1};
    \node (2) at (0,2) {1};
    \node[new] (3) at (-0.75,4) {2};
    \node[new] (4) at (0.75,4) {2};
    \node[new] (5) at (4,0) {2};
    \node[new] (6) at (4,1.5) {2};
    \node[new] (7) at (1.75,2.5) {2};
    \node[new] (8) at (2.5,1.75) {2};
    \path[->][edge]  (root) to (1);
    \path[->][edge]  (root) to (2);
    \path[->][edge_recent]  (2) to (3);
    \path[->][edge_recent]  (2) to (4);
    \path[->][edge_recent]  (1) to (5);
    \path[->][edge_recent]  (1) to (6);
    \path[->][edge_recent]  (root) to (7);
    \path[->][edge_recent]  (root) to (8);
\end{tikzpicture}
\hspace{1cm}
\end{center}
\caption{Action Graphs $G_0$ through $G_2$}
\label{SuperCatFigure}
\end{figure}
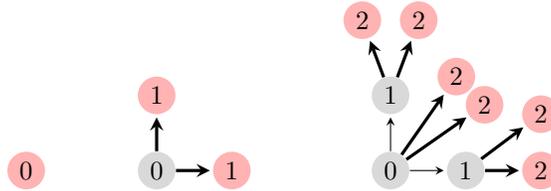
The construction of generalized action graphs for the super Catalan numbers is similar to the generalized action graphs for the Fuss-Catalan numbers, where the number of vertices added adjacent to a vertex is dependent on path length.  For the action graphs related to the super Catalan numbers, $\frac{2}{2^l}$ new vertices are added per path. When $\ell=1$: $$\frac{2}{2^1}=\frac{2}{2}=1.$$
More generally for $\ell=0,1,2,3\ldots$, the number of vertices added to the start of each path is scaled by $[2, 1, \frac{1}{2}, \frac{1}{4}\ldots]$. This can be seen in Figure \ref{SuperCatFigure}, in $G_2$, where 2 vertices are added to each of the vertices labeled 1, for paths of length 0, and 1 vertex is added to the 0 vertex of path of length 1.  

Initially, one might be concerned that this scaling by $\frac{2}{2{\ell}}$ would result in adding fractional vertices. However, it turns out that the number of paths $p(v,\ell)$ will have a large enough power of 2 as a factor to avoid this. 

Caldwell et al. gave the construction of directed graphs in Definition \ref{SuperCatAGConstruction} and conjectured that the graphs are generalized action graphs for $S_{0,n}$, the super Catalan numbers where $m=0$. Theorem \ref{Mtheorem} answers this conjecture. The following example applies Theorem \ref{Mtheorem} to the super Catalan numbers. This example will consider the generalized action graph $G_2$ as seen in Figure \ref{SuperCatFigure}. \\
 \begin{example}
    From Definition \ref{SuperCatalnDef}, when $m = 0$, $S_{0,n} = [1, 2 ,6, 20 \ldots]$. From Figure \ref{SuperCatFigure}, the $z_n$ values for the generalized action graph $G_2$, are $z_n = [2, 2, 4 \ldots]$. In order to apply Theorem \ref{Mtheorem}, we must check that $$S_{0,n} = z_n + \sum_{i=1}^{n-1}z_i \cdot S_{(0,n-i)}.$$ When $n=2:$
    \begin{align}
        \notag S_{0,2} &= z_2 + z_1 \cdot S_{0,1} \\ \notag
        &= 2 + 2 \cdot 2\\ \notag 
        &= 6.
    \end{align}
    Since $S_{0,2}= 6$, our theorem implies the existence of generalized action graphs for the super Catalan numbers when $m=0$ and $n\leq 2$. 
\end{example}

Because these graphs can be cumbersome to draw after the first few graphs, the remaining graphs will be drawn using a condensed notation developed in \cite{CURMGroup}.

\begin{notation}[\cite{CURMGroup}]\label{notation}
Since generalized action graphs have many identical subgraphs with the same labels, we collapse them. For an edge from a vertex labeled $a$ to a vertex labeled $b$, the multiplier $\times m$, indicates the number of such edges from a vertex labeled $a$ to vertices labeled $b$ it corresponds to in the original graph. For a vertex in the condensed form, we can find the number of vertices it represents in the standard form by multiplying the labels along the path from the root to that vertex. For example, the upper right vertex labeled 2 in the condensed graph in figure 3 represents the $2 \times 2 = 4$ vertices labeled 2 that are adjacent to the two vertices labeled 1 in the original graph. 
\end{notation}

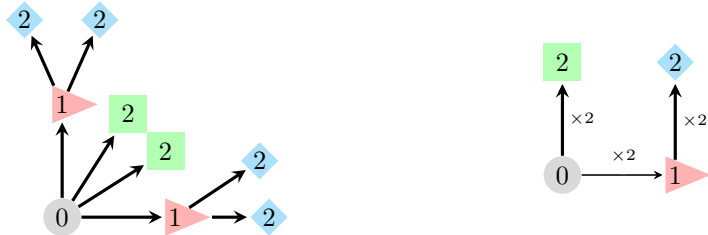
\begin{minipage}{0.5\textwidth}
\begin{center}
    \begin{tikzpicture}[> = stealth,shorten > = 1pt,scale=.5]
    \node (root) at (0,0) {0};
    \node[new, isosceles triangle] (1) at (3,0) {1};
    \node[new, isosceles triangle] (2) at (0,3) {1};
    \path[->][edge_recent]  (root) to (1);
    \path[->][edge_recent]  (root) to (2);
    \node[cyan, diamond] (3) at (5.25,1.5) {2};
    \node[cyan, diamond] (4) at (5.5,0) {2};
    \node[cyan, diamond] (5) at (1, 5.25) {2};
    \node[cyan, diamond] (6) at (-1, 5.25) {2};
    \path[->][edge_recent]  (1) to (3);
    \path[->][edge_recent]  (1) to (4);
    \path[->][edge_recent]  (2) to (5);
    \path[->][edge_recent]  (2) to (6);
    \node[green, rectangle] (7) at (2.75, 1.75) {2};
    \node[green, rectangle] (8) at (1.75, 2.75) {2};
    \path[->][edge_recent]  (root) to (7);
    \path[->][edge_recent]  (root) to (8);
\end{tikzpicture}
\end{center}
\end{minipage}
\begin{minipage}{0.5\textwidth}
\begin{center}
    \begin{tikzpicture}[> = stealth,shorten > = 1pt,scale=.5]
    \node (root) at (0,0) {0};
    \node [new, isosceles triangle] (1) at (3,0) {1};
    \node[green, rectangle] (2) at (0,3) {2};
    \node[cyan, diamond] (4) at (3,3) {2};
{\tiny
    \draw[edge, ->] (root) -- (1) node[label, midway, above] {$\times 2$};
    \draw[edge, ->] (1) -- (4) node[label, midway, right] {$\times 2$};
    \draw[edge, ->] (root) -- (2) node[label, midway, right] {$\times 2$};
} 
    \path[->][edge] (root) to (1);
    \path[->][edge_recent] (root) to (2);
    \path[->][edge_recent] (1) to (4);
\end{tikzpicture}
\end{center}
\end{minipage}

\section{Generalized Action Graph Properties} \label{Sec: AGP}
 Action graphs were originally defined in the work of Bergner and Hackney on Reedy categories  \cite{Reedy}. Then Alvarez, Bergner and Lopez in \cite{ActionGraphs}, Caldwell et al. in \cite{CURMGroup}, and Cressman et al. in \cite{GenActionGraphs}, showed correspondences between these action graphs and sequences related to the Catalan numbers. Through the work of \cite{CURMGroup}, similar properties were found for all such generalized action graphs, regardless of the sequences with which the graphs are affiliated. We will assume these properties hold for the construction of all generalized action graphs. That is, when constructing action graph ${G_n}$ for sequence ${s_n}$, we will assume they satisfy the following definition.
 \begin{definition}[\cite{CURMGroup}]\label{axiomdef}
     The sequence ${G_n}$ of generalized action graphs for a particular sequence ${s_n}$ of positive integers is a sequence of directed, labeled graphs such that:
 \end{definition}
 \begin{enumerate}[label= Axiom \arabic{enumi}, leftmargin=*]
     \item \label{Axiom1} Define $G_0$ as the graph with $s_0$ vertices labeled 0 and no edges. Construct $G_n$ from $G_{n-1}$ by adding $s_n$ new vertices, which are each labeled $n$.
     \item \label{Axiom2} For vertex $v$ in $G_n$, the subtree of $G_n$ with root $v$ is isomorphic to some $G_k$ such that $k \leq n$.
     \item \label{Axiom3} All leaves in the graph $G_n$ have label $n$.
 \end{enumerate}

When constructing generalized action graphs for the super Catalan numbers, the authors in \cite{CURMGroup} also proved Lemmas \ref{AG_One} and \ref{squared}, as necessary properties for generalized action graphs. 

 \begin{lemma}[\cite{CURMGroup}] \label{AG_One}
     In order for a sequence $\{s_n\}_{n=0}$ to form a valid generalized action graph, $s_n$ must have the property that $s_0 = 1$.   
 \end{lemma}
 \begin{lemma}[\cite{CURMGroup}]\label{squared}
    In order for a sequence $\{s_n\}_{n=0}$ to form a valid generalized action graph, $s_n$ must have the property $s_2\geq s_{1}^2$.
\end{lemma}
An important aspect of Theorem \ref{Mtheorem} relies on \ref{Axiom2}, regarding graph isomorphism. 

\subsection{Graph Isomorphism}
Graphs are \emph{isomorphic} to each other if they have the same number of vertices and edges, connected in the same way. From \ref{Axiom2}, generalized action graphs have isomorphic subtrees, or subgraphs, meaning that previous graphs can be seen in subsequent graphs. Note that when we consider isomorphism in generalized action graphs, the labels on the vertices may be shifted, but the underlying structure of the graphs are the same. We can see this below using an action graph for the super Catalan numbers, $G_3$, in condensed notation.
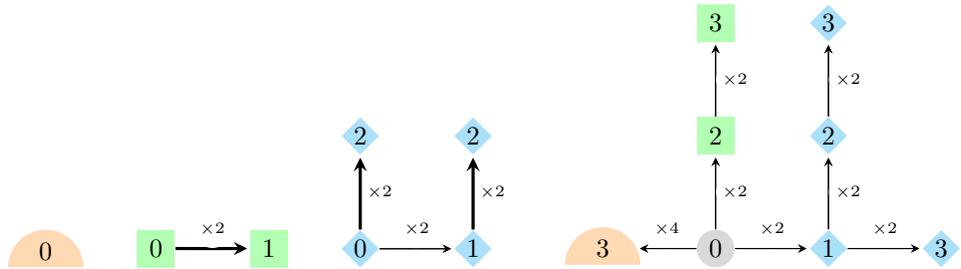
\begin{figure}[H]
\begin{center}
\begin{tikzpicture}[> = stealth,shorten > = 1pt,scale=.5]
    \node[orange,semicircle] (root) at (0,0) {0};
\end{tikzpicture}
\hspace{0.5cm}
\begin{tikzpicture}[> = stealth,shorten > = 1pt,scale=.5]
    \node[green,rectangle] (root) at (0,0) {0};
    \node[green,rectangle] (1) at (3,0) {1};
    \path[->][edge_recent]  (root) to (1);

    {\tiny
    \draw[edge, ->] (root) -- (1) node[label, midway, above] {$\times 2$};
    }
\end{tikzpicture}
\hspace{0.5cm}
\begin{tikzpicture}[> = stealth,shorten > = 1pt,scale=.5]
    \node[cyan,diamond] (root) at (0,0) {0};
    \node[cyan,diamond] (1) at (3,0) {1};
    \node[cyan,diamond] (2) at (0,3) {2};
    \node[cyan,diamond] (4) at (3,3) {2};
{\tiny
    \draw[edge, ->] (root) -- (1) node[label, midway, above] {$\times 2$};
    \draw[edge, ->] (1) -- (4) node[label, midway, right] {$\times 2$};
    \draw[edge, ->] (root) -- (2) node[label, midway, right] {$\times 2$};
} 
    \path[->][edge] (root) to (1);
    \path[->][edge_recent] (root) to (2);
    \path[->][edge_recent] (1) to (4);
\end{tikzpicture}
\hspace{0.5cm}
\begin{tikzpicture}[> = stealth,shorten > = 1pt,scale=.5]
    \node (root) at (0,0) {0};
    \node[cyan,diamond] (1) at (3,0) {1};
    \node[green,rectangle] (2) at (0,3) {2};
    \node[green,rectangle] (3) at (0,6) {3};
    \node[cyan,diamond] (4) at (3,3) {2};
    \node[cyan,diamond] (5) at (3,6) {3};
    \node[orange,semicircle] (6) at (-3,0) {3};
    \node[cyan,diamond] (7) at (6,0) {3};

    \path[->][edge] (root) to (1);
    \path[->][edge] (root) to (2);
    \path[->][edge] (2) to (3);
    \path[->][edge] (1) to (4);
    \path[->][edge] (4) to (5);
    \path[->][edge] (root) to (6);
    \path[->][edge] (1) to (7);

{\tiny
    \draw[edge, ->] (root) -- (1) node[label, midway, above] {$\times 2$};
    \draw[edge, ->] (1) -- (4) node[label, midway, right] {$\times 2$};
    \draw[edge, ->] (4) -- (5) node[label, midway, right] {$\times 2$};
    \draw[edge, ->] (root) -- (2) node[label, midway, right] {$\times 2$};
    \draw[edge, ->] (2) -- (3) node[label, midway, right] {$\times 2$};
    \draw[edge, ->] (root) -- (6) node[label, midway, above] {$\times 4$};
    \draw[edge, ->] (1) -- (7) node[label, midway, above] {$\times 2$};
}
\end{tikzpicture}
\end{center}
\caption{Super Catalan Action Graphs $G_0$ through $G_3$}
\label{subgraphfigure}
\end{figure}
Figure \ref{subgraphfigure} shows that each subtree of $G_3$ is isomorphic to a previous action graph, using colored vertices. The number of isomorphic subtrees are indicated by the multiplier on the edges connected to the 0 vertex. For example, $G_3$ has two subtrees isomorphic to $G_2$, two subtrees isomorphic to $G_1$, and four subtrees isomorphic to $G_0$.

\section{Building Generalized Action Graphs} \label{Sec: BAG}
There was previously no general solution for building generalized action graphs based on any given number sequence. Previous progress in \cite{CURMGroup}, \cite{ActionGraphs}, and \cite{GenActionGraphs} defined action graphs and described their properties for specific sequences. Cochran \cite{AliPoster} conjectured a way to generalize this work for any sequence that satisfies a certain property. In particular, she suggested a way to determine the number of new vertices labeled $n$ adjacent to the 0 vertex. Said another way, she conjectured the value of the multiplier on the edge from 0 to the vertex labeled $n$ in condensed notation. 
\begin{conjecture}[\cite{AliPoster}] \label{AliConjecture} 
    Given an appropriate sequence, $s_n$, of positive integers, in order to build action graph $G_n$, the number of new vertices that must be added to the 0 vertex is $$z_n = s_n - \sum_{i=1}^{n-1}z_i\cdot s_{n-i}.$$
\end{conjecture}
Using the properties from Definition \ref{axiomdef} and inspired by Conjecture \ref{AliConjecture}, we arrive at the following theorem.  
\Mtheorem*

\begin{proof}
   Let $s_n$ be a positive sequence for $n\geq 0$ such that $s_0 = 1$ and there exist positive integers $z_n$ such that $z_1 = s_1$ and for all $n\geq 2$,  
    \begin{equation}\label{formula}
        s_n= z_n + \sum_{i=1}^{n-1}z_i\cdot s_{n-i}.
    \end{equation}
By definition of generalized action graph, since $s_0 =1$, $G_0$ is a single vertex labeled $0$. The construction of the remaining graphs will proceed by induction. We include more base cases than necessary in order to illustrate the construction. 
    
    \noindent\underline{Base Case:} Suppose $n=1$.
    
    By assumption, $s_1=z_1$. Adding $s_1=z_1$ new vertices labeled 1 to the previous graph, we arrive at Figure \ref{ProofG_1}.
    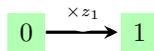
\begin{figure}[H]
    \begin{center}
    \begin{tikzpicture}[> = stealth,shorten > = 1pt,scale=.5]
        \node[green, rectangle] (root) at (0,0) {0};
        \node[green, rectangle] (1) at (3,0) {1};
        \path[->][edge_recent]  (root) to (1);
    
        {\tiny
        \draw[edge, ->] (root) -- (1) node[label, midway, above] {$\times z_1$};
        }
    \end{tikzpicture}
    \end{center}
    \caption{Construct $G_1$ by adding $z_1$ vertices adjacent to the root vertex.}
    \label{ProofG_1}
    \end{figure}
From Figure \ref{ProofG_1}, we can see that $z_1$ vertices labeled $1$ are added to the vertex labeled $0$. Since $z_1 = s_1$, we have added a total number of $s_1$ vertices labeled $1$ to the previous graph, making this graph $G_1$.
    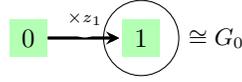
\begin{figure}[H]
    \begin{center}
    \begin{tikzpicture}[> = stealth,shorten > = 1pt,scale=.5]
        \node[green, rectangle] (root) at (0,0) {0};
        \node[green, rectangle] (1) at (3,0) {1};
        \node[rectangle, fill=none] at (5, 0) {\small{$\cong G_0$}};
        \path[->][edge_recent]  (root) to (1);
        \node[ellipse,
        draw = black,
        fill = none,
        minimum width = 1cm,
        minimum height = 1cm] (e) at (3,0) {};
        {\tiny
        \draw[edge, ->] (root) -- (1) node[label, midway, above] {$\times z_1$};
        }
    \end{tikzpicture}
    \end{center}
    \caption{Isomorphic subtrees in $G_1$}
    \label{IsoG_1}
    \end{figure}
    Figure \ref{IsoG_1} shows that the subtree with root vertex 1 is isomorphic to the generalized action graph $G_0$.  In Figure \ref{ProofG_1}, we added $s_1$ vertices to the generalized action graph $G_0$, with $z_1$ vertices adjacent to the 0 vertex, and each subtree with root vertex adjacent to the 0 vertex is isomorphic to some action graph $G_{n-t}$ where $t \leq n$, in this case, $G_0$. Note that all leaves in this graph are labeled $1$. Thus $G_1$ satisfies the definition of generalized action graph given in Definition \ref{axiomdef}. 
    
    \noindent \underline{Base Case:} Suppose $n=2$.

    By plugging in $n=2$ into Formula \ref{formula}, we see $z_2$ must satisfy
    \begin{align}
        s_2 &= z_2 + \sum_{i=1}^{2-1} z_i \cdot s_{2-i} \notag \\
        &= z_2 + z_1\cdot s_1. \notag
    \end{align}
    By assumption, we know $s_1 = z_1$, so $s_2 = z_2 + z_1 \cdot z_1$, and $z_2 = s_2 - z_1z_1$.

    To construct $G_2$, we will add a total of $z_1$ vertices labeled 2 off of the vertex labeled 1 in order to make the subtrees rooted at vertices labeled 1 isomorphic to graph $G_1$. We will also add $z_2$ vertices labeled 2 adjacent to the vertex labeled 0.
    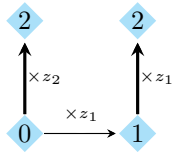
\begin{figure}[H]
    \begin{center}
    \begin{tikzpicture}[> = stealth,shorten > = 1pt,scale=.5]
        \node[cyan, diamond] (root) at (0,0) {0};
        \node[cyan, diamond] (1) at (3,0) {1};
        \node[cyan, diamond] (2) at (0,3) {2};
        \node[cyan, diamond] (4) at (3,3) {2};
    {\tiny
        \draw[edge, ->] (root) -- (1) node[label, midway, above] {$\times z_1$};
        \draw[edge, ->] (1) -- (4) node[label, midway, right] {$\times z_1$};
        \draw[edge, ->] (root) -- (2) node[label, midway, right] {$\times z_2$};
    } 
        \path[->][edge] (root) to (1);
        \path[->][edge_recent] (root) to (2);
        \path[->][edge_recent] (1) to (4);
    \end{tikzpicture}
    \end{center}
    \caption{Construct $G_2$ by adding $s_2$ vertices to $G_1$}
    \label{ProofG_2}
    \end{figure}
Figure \ref{ProofG_2} shows $z_2$ vertices labeled $2$ are added to the vertex labeled $0$, and $z_1 \times z_1$ vertices labeled $2$ are added to the vertices labeled $1$. From this, we know a total of $s_2$ vertices labeled $2$ have been added to the previous generalized action graph, showing that  $G_2$ satisfies \ref{Axiom1} of the definition of generalized action graph.  Since $z_1$ is assumed to be positive, all leaves are labeled 2, and thus $G_2$ also satisfies \ref{Axiom3}.  
    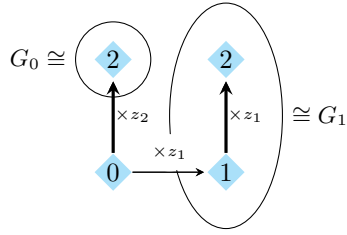
\begin{figure}[H]
    \begin{center}
    \begin{tikzpicture}[> = stealth,shorten > = 1pt,scale=.5]
        \node[cyan, diamond] (root) at (0,0) {0};
        \node[cyan, diamond] (1) at (3,0) {1};
        \node[cyan, diamond] (2) at (0,3) {2};
        \node[cyan, diamond] (4) at (3,3) {2};
        \node[rectangle, fill=none] at (5.5,1.5) {\small{$\cong G_1$}};
        \node[rectangle, fill=none] at (-2, 3) {\small{$G_0 \cong$}};
        \node[ellipse,
        draw = black,
        fill = none,
        minimum width = 1.5cm,
        minimum height = 3cm] (e) at (3,1.5) {};
        \node[ellipse,
        draw = black,
        fill = none,
        minimum width = 1cm,
        minimum height = 1cm] (e) at (0,3) {};
    {\tiny
        \draw[edge, ->] (root) -- (1) node[label, midway, above] {$\times z_1$};
        \draw[edge, ->] (1) -- (4) node[label, midway, right] {$\times z_1$};
        \draw[edge, ->] (root) -- (2) node[label, midway, right] {$\times z_2$};
    } 
        \path[->][edge] (root) to (1);
        \path[->][edge_recent] (root) to (2);
        \path[->][edge_recent] (1) to (4);
    \end{tikzpicture}
    \end{center}
    \caption{Isomorphic subtrees in $G_2$}
    \label{IsoG_2}
    \end{figure}    
    Figure \ref{IsoG_2} illustrates that  subtrees with root vertex 1 are isomorphic to the generalized action graph $G_1$ and leaves are isomorphic to $G_0$. So, we have checked \ref{Axiom3}, and conclude $G_2$ satisfies the definition of generalized action graph. 
    
    \noindent\underline{Base Case:} Suppose $n=3$. 
    
    By plugging in $n=3$ into Formula \ref{formula},  we have
    \begin{align}
        s_3 &= z_3 + \sum_{i=1}^{3-1} z_i \cdot s_{3-i} \notag \\
        &= z_3 + z_1 \cdot s_2 + z_2 \cdot s_1. \notag
    \end{align}
    By assumption $s_1 = z_1$, and from base case $n=2$ we know $s_2 = z_2 + z_1 \cdot z_1$. So equivalently
    \begin{align}
        s_3 &= z_3 + z_1(z_2 + z_1 \cdot z_1) + z_2 \cdot z_1 \notag\\
        &=z_3 + z_1 \cdot z_2 + z_1 \cdot z_1 \cdot z_1 + z_2 \cdot z_1. \notag
    \end{align}

    We use the values $z_1$, $z_2$, and $z_3$ to construct $G_3$ from $G_2$ as seen in Figure \ref{ProofG_3}.
    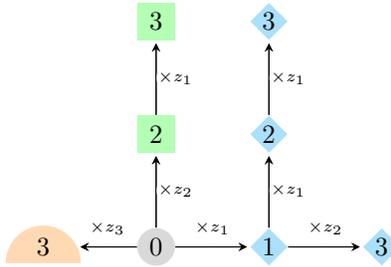
\begin{figure}[H]
    \begin{center}
    \begin{tikzpicture}[> = stealth,shorten > = 1pt,scale=.5]
        \node (root) at (0,0) {0};
        \node[cyan, diamond] (1) at (3,0) {1};
        \node[green, rectangle] (2) at (0,3) {2};
        \node[green, rectangle] (3) at (0,6) {3};
        \node[cyan, diamond] (4) at (3,3) {2};
        \node[cyan, diamond] (5) at (3,6) {3};
        \node[orange, semicircle] (6) at (-3,0) {3};
        \node[cyan, diamond] (7) at (6,0) {3};
    
        \path[->][edge] (root) to (1);
        \path[->][edge] (root) to (2);
        \path[->][edge] (2) to (3);
        \path[->][edge] (1) to (4);
        \path[->][edge] (4) to (5);
        \path[->][edge] (root) to (6);
        \path[->][edge] (1) to (7);
    
    {\tiny
        \draw[edge, ->] (root) -- (1) node[label, midway, above] {$\times z_1$};
        \draw[edge, ->] (1) -- (4) node[label, midway, right] {$\times z_1$};
        \draw[edge, ->] (4) -- (5) node[label, midway, right] {$\times z_1$};
        \draw[edge, ->] (root) -- (2) node[label, midway, right] {$\times z_2$};
        \draw[edge, ->] (2) -- (3) node[label, midway, right] {$\times z_1$};
        \draw[edge, ->] (root) -- (6) node[label, midway, above] {$\times z_3$};
        \draw[edge, ->] (1) -- (7) node[label, midway, above] {$\times z_2$};
    }
    \end{tikzpicture}
    \end{center}
    \caption{Construct $G_3$ by adding $s_3$ vertices $G_2$}
    \label{ProofG_3}
    \end{figure}
Figure \ref{ProofG_3} shows $z_3$ vertices labeled $3$ adjacent to the vertex labeled $0$, as well as $z_1$ vertices labeled $3$ adjacent to each vertex labeled $2$, and $z_2$ vertices labeled $3$ adjacent to each vertex labeled $1$. From this we know the total number of vertices added is $s_3$, making this graph $G_3$. Thus $G_3$ satisfies \ref{Axiom1} of the definition of generalized action graph. Since all leaves in $G_2$ were labeled 2 and $z_1$ is positive, we see all leaves in $G_3$ are labeled 3, satisfying \ref{Axiom3}. 
    \begin{figure}[H]
    \begin{center}
    \begin{tikzpicture}[> = stealth,shorten > = 1pt,scale=.5]
        \node (root) at (0,0) {0};
        \node[cyan, diamond] (1) at (3,0) {1};
        \node[green, rectangle] (2) at (0,3) {2};
        \node[green, rectangle] (3) at (0,6) {3};
        \node[cyan, diamond] (4) at (3,3) {2};
        \node[cyan, diamond] (5) at (3,6) {3};
        \node[orange, semicircle] (6) at (-3,0) {3};
        \node[cyan, diamond] (7) at (6,0) {3};
    
        \path[->][edge] (root) to (1);
        \path[->][edge] (root) to (2);
        \path[->][edge] (2) to (3);
        \path[->][edge] (1) to (4);
        \path[->][edge] (4) to (5);
        \path[->][edge] (root) to (6);
        \path[->][edge] (1) to (7);
    
        \node[ellipse, draw = black, fill = none, minimum width = 3cm, minimum height = 5cm] (e) at (4.5,3) {};
        \node[ellipse, draw = black, fill = none, minimum width = 1.25cm, minimum height = 3cm] (e) at (0,4.5) {};
        \node[ellipse, draw = black, fill = none, minimum width = 1cm, minimum height = 1cm] (e) at (-3,0) {};
        \node[rectangle, fill=none] at (8.5,3) {\small{$\cong G_2$}};
        \node[rectangle, fill=none] at (-2.25,4.5) {\small{$G_1 \cong$}};
        \node[rectangle, fill=none] at (-5,0) {\small{$G_0 \cong$}};
    
    {\tiny
        \draw[edge, ->] (root) -- (1) node[label, midway, above] {$\times z_1$};
        \draw[edge, ->] (1) -- (4) node[label, midway, right] {$\times z_1$};
        \draw[edge, ->] (4) -- (5) node[label, midway, right] {$\times z_1$};
        \draw[edge, ->] (root) -- (2) node[label, midway, right] {$\times z_2$};
        \draw[edge, ->] (2) -- (3) node[label, midway, right] {$\times z_1$};
        \draw[edge, ->] (root) -- (6) node[label, midway, above] {$\times z_3$};
        \draw[edge, ->] (1) -- (7) node[label, midway, above] {$\times z_2$};
    }
    \end{tikzpicture}
    \end{center}
    \caption{Isomorphic subtrees in $G_3$}
    \label{IsoG_3}
    \end{figure}
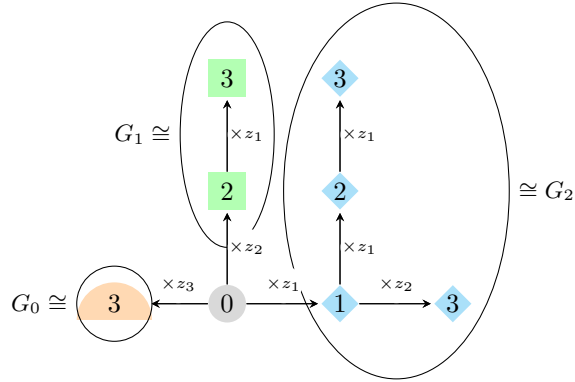
    Figure \ref{IsoG_3} illustrates that subtrees with root vertex labeled 1 are isomorphic to $G_2$, subtrees with root vertex labeled 2 are isomorphic to $G_1$, and subtrees with root vertex labeled 3 are isomorphic to $G_0$. So, we have check \ref{Axiom2}, and conclude $G_3$ satisfies \ref{axiomdef}.
    
    \noindent\underline{Inductive Hypothesis:} Consider the case where $n=k\geq 0$. In particular, assume that for some $k \geq 0$, there exists a sequence of generalized action graphs $G_0, \dots, G_k$ that are constructed by adding $s_i$ vertices to the action graph $G_{i-1}$, that there are $z_i$ vertices labeled $i$ adjacent to the vertex labeled 0, and each subtree in $G_k$ with the root vertex labeled $t$ adjacent to the 0 vertex is isomorphic to $G_{k-t}$, where $t \leq k$, with all leaves labeled $k$.

    Since $k$ is arbitrary, there is no way to draw these action graphs in their completeness. However, by induction, it follows that for each vertex $v$ labeled $i$ adjacent the vertex labeled 0, the subtree with root $v$ is isomorphic to $G_{k-i}$, and we will use a dashed edge to represent the remainder of the subtree.
    \begin{figure}[H]
    \begin{center}
    \begin{tikzpicture}[> = stealth,shorten > = 1pt,scale=.5]
        \node (root) at (0,0) {0};
        \node[new, diamond] (1) at (4,0) {1};
        \node[rectangle, fill=none] at (2,-0.5) 
        {\small{$\times z_1$}};
        \node[cyan, isosceles triangle] (2) at (2.83,2.83) {2};
        \node[rectangle, fill=none] at (2,1.25) 
        {\small{$\times z_2$}};
        \node [purple, trapezium] (3) at (0, 4) {3};
        \node[rectangle, fill=none] at (0.6,2) 
        {\small{$\times z_3$}};
        \node[green, rectangle] (5) at (-2.83,2.83) {k-1};
        \node[rectangle, fill=none] at (-2.15, 1) {\small{$z_{k-1} \times$}};
        \node[orange, semicircle] (4) at (-4,0) {$k$};
        \node[rectangle, fill=none] at (-1.5,-0.5) 
        {\small{$\times z_k$}};
        \node[green, rectangle] (6) at (-2.83, 5.6) {$k$};
        \node[rectangle, fill=none] at (-2, 4.25) {\small{$\times z_1$}};
        \node[rectangle, fill=none] at (-0.5,2.8) {\small{.}};
        \node[rectangle, fill=none] at (-1.6,2.25) {\small{.}};
        \node[rectangle, fill=none] at (-1.15,2.7) {\small{.}};
    
        \path[->][edge] (root) to (1);
        \path[->][edge] (root) to (2);
        \path[->][edge] (root) to (3);
        \path[->][edge] (root) to (4);
        \path[->][edge] (root) to (5);
        \path[->][edge] (5) to (6);
        \draw[->,thick,black,dashed] (4.75,0) -- (6.75,0) {};
        \draw[->,thick,black,dashed] (2.83,3.75) -- (2.83,5.75) {};
        \draw[->,thick,black,dashed] (0,4.75) -- (0,6.75) {};
        
    \end{tikzpicture}
    \end{center}
    \caption{Graph $G_k$}
    \label{G_k}
    \end{figure}
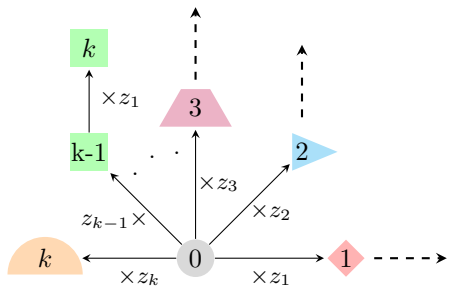
Figure \ref{G_k} shows that there are $z_k$ vertices labeled $k$ adjacent to the vertex labeled $0$. By the induction hypothesis, we also know that there are a total of $s_k$ vertices labeled $k$ in the graph $G_k$.   
    \begin{figure}[H]
    \begin{center}
    \begin{tikzpicture}[> = stealth,shorten > = 1pt,scale=.5]
        \node (root) at (0,0) {0};
        \node[new, diamond] (1) at (4,0) {1};
        \node[rectangle, fill=none] at (2,-0.5) 
        {\small{$\times z_1$}};
        \node[cyan, isosceles triangle] (2) at (2.83,2.83) {2};
        \node[rectangle, fill=none] at (2,1.25) 
        {\small{$\times z_2$}};
        \node [purple, trapezium] (3) at (0, 4) {3};
        \node[rectangle, fill=none] at (0.6,2) 
        {\small{$\times z_3$}};
        \node[green, rectangle] (5) at (-2.83,2.83) {k-1};
        \node[rectangle, fill=none] at (-2.15, 1) {\small{$z_{k-1} \times$}};
        \node[orange, semicircle] (4) at (-4,0) {$k$};
        \node[rectangle, fill=none] at (-1.5,-0.5) 
        {\small{$\times z_k$}};
        \node[green, rectangle] (6) at (-2.83, 5.6) {$k$};
        \node[rectangle, fill=none] at (-2.15, 4.25) {\small{$\times z_1$}};
        \node[rectangle, fill=none] at (-0.5,2.8) {\small{.}};
        \node[rectangle, fill=none] at (-1.6,2.25) {\small{.}};
        \node[rectangle, fill=none] at (-1.15,2.7) {\small{.}};
        \node[rectangle, fill=none] at (8.25,0) {\small{$\cong G_{k-1}$}};
        \node[rectangle, fill=none] at (5,4) 
        {\small{$\cong G_{k-2}$}};
        \node[rectangle, fill=none] at (2,6.5) 
        {\small{$\cong G_{k-3}$}};
        \node[rectangle, fill=none] at (-6,0) {\small{$G_0\cong$}};
        \node[rectangle, fill=none] at (-5, 4.25) {\small{$G_1\cong$}};
    
        \path[->][edge] (root) to (1);
        \path[->][edge] (root) to (2);
        \path[->][edge] (root) to (3);
        \path[->][edge] (root) to (4);
        \path[->][edge] (root) to (5);
        \path[->][edge] (5) to (6);
    
        \node[ellipse, draw = black, fill = none, minimum width = 2cm, minimum height = 1cm] (e) at (5,0) {};
        \draw[->,thick,black,dashed] (4.75,0) -- (6.75,0) {};
        \node[ellipse, draw = black, fill = none, minimum width = 1cm, minimum height = 2cm] (e) at (2.83,4) {};
        \draw[->,thick,black,dashed] (2.83,3.75) -- (2.83,5.75) {};
        \node[ellipse, draw = black, fill = none, minimum width = 1cm, minimum height = 2cm] (e) at (0,5) {};
        \draw[->,thick,black,dashed] (0,4.75) -- (0,6.75) {};
        \node[ellipse, draw = black, fill = none, minimum width = 1cm, minimum height = 1cm] (e) at (-4,0) {};
        \node[ellipse, draw = black, fill = none, minimum width = 1.25cm, minimum height = 2.5cm] (e) at (-2.83,4.25) {};

    \end{tikzpicture}
    \end{center}
    \caption{Isomorphic subtrees in $G_k$}
    \label{isoG_k}
    \end{figure}
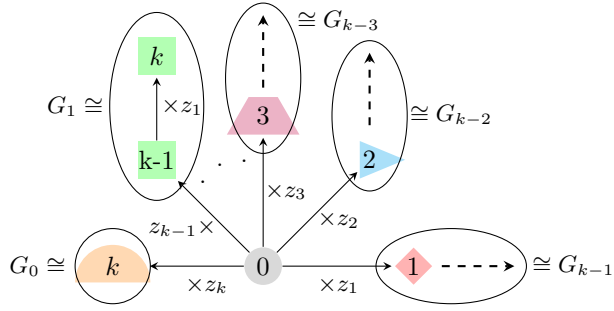
   Figure \ref{isoG_k} shows that each subtree with root vertex adjacent to the vertex labeled 0 is isomorphic to some $G_{k-t}$ where $t\leq k$. 
    
    \underline{Inductive Step:} 
    We will construct $G_{k+1}$ using $G_k$. 

\begin{figure}[H]
\begin{center}
\begin{tikzpicture}[> = stealth,shorten > = 1pt,scale=.5]
    \node (root) at (0,0) {0};
    \node[blue, trapezium] (1) at (4,0) {1};
    \node[rectangle, fill=none] at (2,-0.5) 
    {\small{$\times z_1$}};
    \node[new, diamond] (2) at (2.83,2.83) {2};
    \node[rectangle, fill=none] at (2,1.25) 
    {\small{$\times z_2$}};
    \node [cyan, isosceles triangle] (3) at (0, 4) {3};
    \node[rectangle, fill=none] at (0.6,2) 
    {\small{$\times z_3$}};
    \node[green, rectangle] (4) at (-2.83,2.83) {$k$};
    \node[rectangle, fill=none] at (-1,1.75) 
    {\small{$\times z_k$}};
    \node[orange, semicircle] (5) at (-4, 0) {$k+1$};
    \node[rectangle, fill=none] at (-1.5, -0.5) 
    {\small{$\times z_{k+1}$}};
    \node[rectangle, fill=none] at (-0.5,2.8) {\small{.}};
    \node[rectangle, fill=none] at (-1.6,2.25) {\small{.}};
    \node[rectangle, fill=none] at (-1.15,2.7) {\small{.}};
    \node [green, rectangle] (6) at (-2.83, 5.6) {$k+1$};
    \node[rectangle, fill=none] at (-2.2, 4) {\small$\times z_1$};

    \path[->][edge] (root) to (1);
    \path[->][edge] (root) to (2);
    \path[->][edge] (root) to (3);
    \path[->][edge] (root) to (4);
    \path[->][edge] (root) to (5);
    \path[->][edge] (4) to (6);
    
    \draw[->,thick,black,dashed] (4.75,0) -- (6.75,0) {};
    \draw[->,thick,black,dashed] (2.83,3.58) -- (2.83,5.58) {};
    \draw[->,thick,black,dashed] (0,4.75) -- (0,6.75) {};

    \node[ellipse, draw = black, fill = none, minimum width = 2cm, minimum height = 1cm] (e) at (5,0) {};
    \node[rectangle, fill=none] at (8.25,0) {\small{$\cong G_k$}};
    \node[ellipse, draw = black, fill = none, minimum width = 1.25cm, minimum height = 1.25cm] (e) at (-4,0) {};
    \node[rectangle, fill=none] at (5,3.83) {\small{$\cong G_{k-1}$}};
    \node[ellipse, draw = black, fill = none, minimum width = 1cm, minimum height = 2cm] (e) at (2.83,3.83) {};
    \node[rectangle, fill=none] at (2,6.5) {\small{$\cong G_{k-2}$}};
    \node[ellipse, draw = black, fill = none, minimum width = 1cm, minimum height = 2cm] (e) at (0,5) {};
    \node[rectangle, fill=none] at (-5,3.83) {\small{$G_1\cong$}};
    \node[ellipse, draw = black, fill = none, minimum width = 1.25cm, minimum height = 2.25cm] (e) at (-2.83,4.2) {};
    \node[rectangle, fill=none] at (-6.25,0) {\small{$G_0 \cong$}};
    
\end{tikzpicture}
\end{center}
\caption{Isomorphic subtrees in $G_{k+1}$.}
\label{IsoG_k+1}
\end{figure}
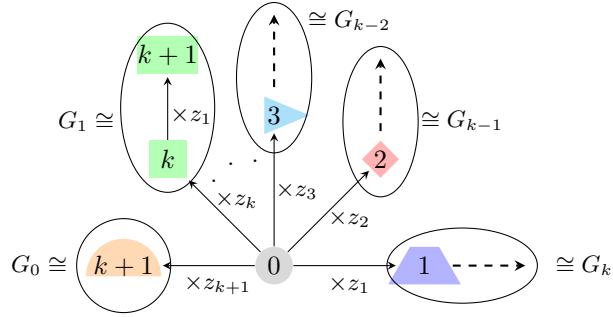

We construct $G_k+1$ in the following way, as illustrated by Figure \ref{IsoG_k+1}. 

We add $z_{k+1}$ vertices, which are each isomorphic to $G_0$, adjacent to the vertex labeled 0.

In general, we consider the vertices labeled $i$ adjacent to the vertex labeled 0 in $G_k$, and those vertices are each roots of subtrees isomorphic to $G_{k-i}$. By the induction hypothesis, we already know how to add new vertices labeled $k-i+1$ to $G_{k-i}$ in order to construct $G_{k-i+1}$. We use this same process to add vertices labeled $k+1$ to the subtrees isomorphic to $G_{k-i}$ in $G_k$. This will result in $z_i$ subtrees isomorphic to $G_{k-i+1}$ with with roots labeled $i$ adjacent to the vertex labeled 0 in $G_{k+1}$. Once we have done this for all $1\leq i\leq k$, the construction is complete. 

For example: In $G_k$, each vertex labeled $k$ adjacent to 0 is the root of a subtree isomorphic to $G_0$. So, to each vertex labeled $k$ adjacent to 0, we add $z_1$ adjacent new vertices. This results in $z_k$ subtrees isomorphic to $G_1$ with roots labeled $k$ adjacent to the 0 vertex in $G_{k+1}$. Note that this step adds $z_k\cdot s_1$ new vertices.

As a further example: In $G_k$, each vertex labeled $k-1$ adjacent to 0 is the root of a subtree isomorphic to $G_1$. To each of these subtrees isomorphic to $G_1$, we add new vertices labeled $k$ in the same way that we added vertices labeled 2 to construct $G_2$ from $G_1$. This will result in $z_{k-1}$ subtrees isomorphic to $G_2$ with roots labeled $k-1$ adjacent to the 0 vertex in $G_{k+1}$. 

By construction and the induction hypothesis, $G_{k+1}$ satisfies \ref{Axiom2}. Since $G_{k+1}$ satisfies \ref{Axiom2}, each leaf in $G_k$ now has $z_1$ new leaves labeled $K+1$ adjacent to it in $G_{k+1}$. Thus, $G_{k+1}$ satisfies \ref{Axiom3}. 

By the induction hypothesis and \ref{Axiom1} of the definition of generalized action graph, this construction adds a total of $z_i\cdot s_{k+1-i}$ new vertices for each $1\leq i\leq k$. Summing over $i$ and including the $z_{k+1}$ vertices labeled $k+1$ adjacent to the 0 vertex, the total number of vertices labeled $k+1$ in $G_{k+1}$ is 
\[z_{k+1}+\sum_{i=1}^{k} z_i\cdot s_{k+1-i},\]

which, by assumption, is equal to $s_{k+1}$. Thus $G_{k+1}$ satisfies the definition of generalized action graph, and our proof by induction is complete. 

\end{proof}

\section{Further Questions}
This project opens several directions for future analysis. Theorem \ref{Mtheorem} provides sufficient conditions to construct action graphs. Theorem \ref{Mtheorem} can also help us determine other sequences that could possibly yield action graphs. By finding other sequences that do yield action graphs, we could further generalize what makes sequences like the Catalan, Fuss-Catalan, and super Catalan numbers unique. Another open question is whether the hypotheses of Theorem \ref{Mtheorem} are in fact necessary conditions for generalized action graphs to exist.


\begin{thebibliography}{99} 

\bibitem{AliPoster}
Ali Cochran. ``Action Graphs for Catalan Sequences'' \textit{MathFest Poster Session} (2023) 

\bibitem{ActionGraphs} Alvarez, Bergner, Lopez. ``Action Graphs and Catalan Numbers.'' \textit{Journal of Integer Sequences} Vol. 18 (2015)
\href{https://arxiv.org/abs/1503.00044v1}{https://arxiv.org/abs/1503.00044v1}

\bibitem{CURMGroup} Caldwell, Cochran, Glisson, Jennings,
McDicken, Proctor. ``Catalan Number Sequences and Action Graphs'' \textit{Mathematics Exchange. Ball State University} Volume 18. (2025)

\bibitem{GenActionGraphs} Danielle Cressman, Jonathan Lin, An Nguyen, and Luke Wiljanen. 
``Generalized action graphs.'' 

\bibitem{Reedy}
Julia E. Bergner and Philip Hackney. ``Reedy categories which encode the notion of category.''
\textit{Fundamenta Mathematicae} 228.3 (2015) p. 193-222.
\href{http://eudml.org/doc/282637}{http://eudml.org/doc/282637}

\bibitem{catalantext} Richard Stanley. ``Catalan Numbers'' \textit{Cambridge University Press, New York} (2015)
\end{thebibliography}
\end{document}